\RequirePackage{ifluatex}

\documentclass{article}

\usepackage{amsmath}
\usepackage{amssymb}
\usepackage{amsthm}
\usepackage{mathtools}
\usepackage[latin1]{inputenc}
     
\usepackage{graphicx}

\usepackage[usenames,dvipsnames]{pstricks}
\usepackage{epsfig}
\usepackage{pst-grad} 
\usepackage{pst-plot} 

\usepackage{ifthen}
\usepackage{color}

\newcommand{\intav}[1]{\mathchoice {\mathop{\vrule width 6pt height 3 pt depth  -2.5pt
\kern -8pt \intop}\nolimits_{\kern -6pt#1}} {\mathop{\vrule width
5pt height 3  pt depth -2.6pt \kern -6pt \intop}\nolimits_{#1}}
{\mathop{\vrule width 5pt height 3 pt depth -2.6pt \kern -6pt
\intop}\nolimits_{#1}} {\mathop{\vrule width 5pt height 3 pt depth
-2.6pt \kern -6pt \intop}\nolimits_{#1}}}

\def\polhk#1{\setbox0=\hbox{#1}{\ooalign{\hidewidth\lower1.5ex\hbox{`}\hidewidth\crcr\unhbox0}}}

\def\Xint#1{\mathchoice
{\XXint\displaystyle\textstyle{#1}}%
{\XXint\textstyle\scriptstyle{#1}}%
{\XXint\scriptstyle\scriptscriptstyle{#1}}%
{\XXint\scriptscriptstyle\scriptscriptstyle{#1}}%
\!\int}
\def\XXint#1#2#3{{\setbox0=\hbox{$#1{#2#3}{\int}$ }
\vcenter{\hbox{$#2#3$ }}\kern-.6\wd0}}

\def\dashint{\Xint-}

\newcommand{\vmo}{\operatorname{VMO}}
\newcommand{\bmo}{\operatorname{BMO}}
\newcommand{\dist}{\operatorname{dist}}

\newcommand{\tr}{\operatorname{Tr}}

\newcommand{\llip}{\operatorname{Log-Lip}}

\newtheorem{teo}{Theorem}[section]

\newtheorem{Definition}{Definition}[section]
\newtheorem{Lemma}{Lemma}[section]

\newtheorem{Proposition}{Proposition}[section]
\newtheorem{Remark}{Remark}[section]
\newtheorem{Assumption}{A}

\begin{document}

\title{Geometric regularity theory for a time-dependent Isaacs equation}
\author{ P\^edra D. S. Andrade, Giane C. Rampasso and Makson S. Santos*}

\date{\today} 

\maketitle

\begin{abstract}

\noindent The purpose of this work is to produce a regularity theory for a class of parabolic Isaacs equations. Our techniques are based on approximation methods which allow us to connect our problem with a Bellman parabolic model. An approximation regime for the  coefficients, combined with a smallness condition on the source term unlocks new regularity results in Sobolev and H\"older spaces.  

\medskip


\medskip 


\end{abstract}

\vspace{.1in}

\section{Introduction}
\label{introduction}
In this paper, we investigate a regularity theory for $L^p$-viscosity solutions to an Isaacs parabolic equation of the form 
\begin{equation} \label{eq_main}
	u_t +  \sup_{\alpha \in {\cal A}} \inf_{\beta \in {\cal B}}\left[ - {\tr}(A_{\alpha, \beta}(x,t)D^{2} u)\right] = f \quad \text{in} \quad Q_1,
\end{equation}
where $Q_1 := B_1\times (-1, 0]$, $A_{\alpha,\beta}:Q_1\times{\cal A}\times{\cal B} \rightarrow \mathbb{R}^{d^2}$ is a $(\lambda, \Lambda)$-elliptic matrix, $\cal A$ and $\cal B$ are countable sets and the source term $f$ satisfies a set of conditions to be specified later. We produce new results on the regularity for the $L^p$-viscosity solutions to \eqref{eq_main}. In particular, we are interested in obtaining improved regularity in Sobolev and H\"older spaces. We argue by approximation methods relating the model in \eqref{eq_main} to a Bellman parabolic problem.

Approximation methods were introduced by L. Caffarelli in the groundbreaking paper \cite{caffarelli}. More recently, these methods appeared in more general contexts, as developed in the works of E. Teixeira, J.M. Urbano and their collaborators; see for instance \cite{M3}, \cite{M8}, \cite{M1}, \cite{M5}, \cite{M4}. We also refer to the surveys \cite{pim_san}, \cite{M7}. Finally, we emphasize that important results have been obtained by approximation methods in a variety of different settings, for instance see \cite{ku_min2}, \cite{ku_min1}, just to mention a few.

The Isaacs equation appears in several branches of applied mathematics. Originally, it was introduced in the works of R. Isaacs, as the PDE associated with two players zero-sum stochastic differential games, see \cite{isaacs}. We notice a revitalization of the interest in this class of equation as the theory of viscosity solution was introduced. We refer the reader to \cite{cra_ev_li}, \cite{cra_li}, \cite{ev_sou}, \cite{fl_sou} for existence and uniqueness of viscosity solutions. In \cite{kat}, the author developed representation formulas for viscosity solutions in the parabolic setting. See also \cite{bucaqui}. 


The study of the regularity theory for fully nonlinear parabolic equations first appeared in \cite{krysaf2}, \cite{KrySaf}. In these papers, the authors examine linear parabolic equations with measurable coefficients. This analysis enables them to produce a Harnack inequality and develop regularity theory of the solutions to fully nonlinear parabolic equations of the form
\begin{equation}\label{eq_fully}
	u_t+F(D^2u)=0 \,\,\, \mbox{in} \,\,\, Q_1.
\end{equation}
Namely, by a linearization argument, viscosity solutions to \eqref{eq_fully} are of class $\mathcal{C}^{1,\gamma}$, for some $\gamma\in(0,1)$. Under convexity assumptions on the operator $F$, the authors in \cite{krylov1}, \cite{krylov2} proved estimates in $\mathcal{C}_{loc}^{2,\gamma}(Q_1)$ for viscosity solutions to \eqref{eq_fully}.

In \cite{wangI}, \cite{wangII} under the assumptions that the operator with frozen coefficients has appropriated interior estimates, the author establishes several a priori estimates in Sobolev and H\"older spaces for fully nonlinear parabolic equations extending the results in \cite{caffarelli}; see also \cite{cafcab}. A Harnack inequality for fully nonlinear uniformly parabolic equations is treated in \cite{imbersil}, where the authors also study existence, uniqueness and regularity results for viscosity solutions. Under an almost convexity assumption, the authors in \cite{Dong-Krylov-2019} prove weighted and mixed-norm Sobolev estimates for fully nonlinear second-order elliptic and parabolic equations with almost $\vmo$ dependence on space-time variables; see also \cite{Krylov-2018}.

The former developments rely on notions pertinent to the realm of $\mathcal{C}$-viscosity solutions, see \cite{ccks1996}. In \cite{cra_ko_swI} a $L^p$-viscosity theory for fully nonlinear parabolic equations is put forward. More precisely, the authors prove $W^{2,1;p}$-estimates when the matrix $A_{\alpha, \beta}$ is independent of $\beta$. They also establish $\mathcal{C}^{1,\gamma}$-estimates in the case that $p > d+2$. In addition, they obtain H\"older regularity estimates for the gradient for $\mathcal{C}$-viscosity solutions.

Regularity estimates for $L^p$-viscosity solutions are also the subject of \cite{M8}. In that paper, the authors establish a regularity theory for this type of problem involving source terms with mixed norms, under distinct regularity regimes. In particular, they prove optimal interior regularity results in $\mathcal{C}^{0,\gamma}(Q_1)$, $\mathcal{C}^{\llip}(Q_1)$ and $\mathcal{C}^{1, \llip}(Q_1)$ spaces. For sharp regularity estimates in Sobolev spaces we refer to \cite{caspim}. Local regularity in ${\mathcal C}^{1,\gamma}$ spaces is also the subject of \cite{kry2} where the assumptions under the coefficients are different from \cite{cra_ko_swI} and \cite{wangI}.

A remarkable feature of fully nonlinear operators is that every model of the form $F(M)$ can be rewritten as
\[
	\inf_{\alpha \in {\cal A}}\sup_{\beta \in {\cal B}}\{A_{\alpha\beta}(M)\},
\]
for some family $A_{\alpha\beta} = a^{\alpha\beta}_{ij}\partial_{ij}$, see \cite{caf_cab2}. Hence we can get information about the solutions to the parabolic problem governed by $F$ by examining the solutions to an associated Isaacs parabolic problem.

Another point of interest in the regularity theory of the Isaacs equation is due to its nature. Namely, Isaacs operators are neither concave nor convex, which places them off the scope of the Evans-Krylov's theory as developed in \cite{evans}, \cite{krylov1}, \cite{krylov2}. As a consequence we have no a priori reason to expect $\mathcal{C}^{2,\gamma}$-interior estimates or the existence of classical solutions. In fact, in \cite{na_vl} the authors exhibited a solution to an elliptic Isaacs equation whose Hessian blows up in an interior point of the domain. As a conclusion it is not reasonable to expect solutions to be more regular than $\mathcal{C}^{1,\gamma}$, if further conditions are not imposed. 

Furthermore, since the Isaacs equations are positively homogeneous of degree 1 with respect to the Hessian, we are not entitled to resort to the concept of \textit{recession function}, as introduced in \cite{sil_tei}; see also \cite{pim_san} and \cite{tei_pim}. In addition, being positively homogeneous of degree 1, immediately we conclude the Isaacs equation is not driven by a differentiable operator, otherwise it would be a linear operator. So the partial regularity results would not be available for this class of problems.

The purpose of this paper is to extend the results in \cite{Pimentel} to the parabolic setting. In \cite{Pimentel}, the author uses approximation methods to relate solutions of the Isaacs elliptic equation to solutions of the Bellman one, under distinct smallness regimes imposed on the coefficients. The author established that viscosity solutions are locally in $W^{2,p}(B_1)$, ${\mathcal C}^{1,\text{Log-Lip}}(B_1)$, or in ${\mathcal C}^{2,\gamma}(B_1)$, depending on the smallness regimes chosen. 

Under certain smallness regime on the matrix $A_{\alpha,\beta}$, we also use approximation methods to connect $L^p$-viscosity solutions to \eqref{eq_main} with the solutions to a parabolic Bellman equation of the form
\[
	u_t+\inf_{\beta\in\mathcal{B}}\left[-\tr(\bar{A}_{\beta}(x,t)D^2u)\right]=0 \,\,\, \mbox{in} \,\,\, Q_1.
\]
Since the Bellman operator is convex with respect to the Hessian, the idea is to import information from the Evans-Krylov theory to our equation.  

In a first moment, we study the equation with dependence on the gradient 	
\begin{equation} \label{equation03}
	u_t +  \sup_{\alpha \in {\cal A}} \inf_{\beta \in {\cal B}}
	\left[ - {\tr}(A_{\alpha, \beta}(x,t) D^{2} u)-{\bf b}_{\alpha,\beta}(x,t)\cdot Du\right] = f \quad \text{in} \quad Q_1,
\end{equation}
where ${\bf b}_{\alpha,\beta}: Q_1\times\mathcal{A}\times\mathcal{B}\rightarrow\mathbb{R}^d$ is a given vector field. In this case, we follow the ideas in \cite{caspim} and \cite{Pimentel} to obtain Sobolev estimates for viscosity solutions to \eqref{equation03}. The proof of this fact follows from $W^{2,1;\delta}$-estimates for \eqref{eq_main}, combining standard measure-theoretical results and properties of $L^p$-viscosity solutions of fully nonlinear parabolic equations. 

We also deal with the borderline case. In fact, in a different approximation regime and source term conditions, we show that $L^p$-viscosity solutions to \eqref{eq_main} are locally in the parabolic Log-Lipschitz space $\mathcal{C}^{1,\llip}(Q_1)$. Finally, if we refine the approximation regime, we improve the previous regularity result by showing ${C}^{2,\gamma}$-estimates at the origin, for some $0<\gamma<1$.

The remainder of this article is structured as follows: in Section 2 we present our main results; we also gather a few facts used throughout the paper and detail our assumptions. In Section 3, we establish improved regularity in Sobolev spaces. Section 4 is devoted to the proof of improved regularity borderline H\"older spaces. We conclude the paper by establishing $\mathcal{C}^{2,\gamma}$-estimates at the origin for \eqref{eq_main}.

\section{Preliminaries}

\subsection{Notations}
In this section we gather some notations which is used in the paper. The \emph{open ball} in $\mathbb{R}^d$ of radius $r$ and centered at the origin is denoted by $B_r$. We also define the \emph{parabolic domain} by
\[
Q_r\coloneqq B_r\times (-r^2,0]\subset\mathbb{R}^{d+1}
\] 
whose \emph{parabolic boundary} is
\[
\partial_pQ_r\coloneqq B_r\times\{t=-r^2\}\cup\partial B_r\times (-r^2,0].
\]
In addition we define 
\[
Q_r(x_0,t_0) := Q_r + (x_0,t_0).
\]
The \emph{parabolic cube} of side $r$ stands for
\[
K_r := [-r,r]^d\times[-r^2,0].
\]
Given $p\in[1,\infty]$, the \emph{parabolic Sobolev space}  $W^{2,1;p}(Q_r)$ is defined by
	\[
	W^{2,1;p}(Q_r)\coloneqq\{u\in L^{p}(Q_r):u_t,\, Du,\, D^2u \in L^p(Q_r)\}
	\]
	endowed with the natural norm
	\[
	\|u\|_{W^{2,1;p}(Q_r)}=\left[\|u\|^p_{L^p(Q_r)}+\|u_t\|^p_{L^p(Q_r)}+\|Du\|^p_{L^p(Q_r)}+\|D^2u\|^p_{L^p(Q_r)}\right]^{\frac{1}{p}}.
	\]
	Hence, we say that $u\in W_{loc}^{2,1;p}(Q_r)$ if $u\in W^{2,1;p}(Q')$ for all $Q'\Subset Q_r$.
	
	In order to define the parabolic H\"older space, we introduce the \emph{parabolic distance} between the points $(x_1,t_1)$ and $(x_2,t_2)$ in $Q_r$ by
	\[
	\dist((x_1,t_1),(x_2,t_2))\coloneqq\sqrt{|x_1-x_2|^2+|t_1-t_2|}.
	\]
	Therefore, a function $u:Q_r\rightarrow\mathbb{R}$ belongs to the \emph{parabolic H\"older space} $\mathcal{C}^{0,\gamma}(Q_r)$ if the following norm 
	\[
	\|u\|_{\mathcal{C}^{0,\gamma}(Q_r)}\coloneqq \|u\|_{L^{\infty}(Q_r)}+[u]_{\mathcal{C}^{0,\gamma}(Q_r)}
	\]
	is finite, where $[u]_{\mathcal{C}^{0,\gamma}(Q_r)}$ denotes the semi-norm
	\[
	[u]_{\mathcal{C}^{0,\gamma}(Q_r)}\coloneqq \sup_{\substack{(x_1,t_1),(x_2,t_2)\in Q_r \\ (x_1,t_1)\neq(x_2,t_2)}}\frac{|u(x_1,t_1)-u(x_2,t_2)|}{\dist((x_1,t_1),(x_2,t_2))^{\gamma}}.
	\]
	This means that $u$ is $\gamma$-H\"older continuous with respect to the spatial variables and $\frac{\gamma}{2}$-H\"older continuous with respect to the time variable. 
	
	Similarly, we say that $u\in \mathcal{C}^{1,\gamma}(Q_r)$ if the spatial gradient $Du(x,t)$ exists in the classical sense for every $(x,t)\in Q_r$ and the norm
	\begin{equation*}
		\begin{aligned}
			\|u\|_{\mathcal{C}^{1,\gamma}(Q_r)}&\coloneqq \|u\|_{L^{\infty}(Q_r)}+\|Du\|_{L^{\infty}(Q_r)}\\
			&+ \sup_{\substack{(x_1,t_1),(x_2,t_2)\in Q_r \\ (x_1,t_1)\neq(x_2,t_2)}}\frac{|u(x_1,t_1)-u(x_2,t_2)-Du(x_1,t_1)\cdot(x_1-x_2)|}{\dist((x_1,t_1),(x_2,t_2))^{1+\gamma}}.
		\end{aligned}
	\end{equation*}
is finite; \emph{i.e.} $Du$ is $\gamma$-H\"older continuous in the spatial variables and $u$ is $\frac{1+\gamma}{2}$-H\"older continuous with respect to the variable $t$. For the borderline case $\gamma = 1$, we say that solutions are of class ${\mathcal C}^{1, 1}(Q_r)$ if $Du$ is Lipschitz continuous with respect to spatial variable and $u$ is also Lipschitz continuous with respect to time variable. Finally, $u\in\mathcal{C}^{2,\gamma}(Q_r)$ if, for all $(x,t)\in Q_r$, the derivative with respect to the temporal variable $u_t(x,t)$ and the spatial Hessian $D^2u(x,t)$ exist in the classical sense and
\begin{equation*}
	\begin{aligned}
		\|u\|_{\mathcal{C}^{2,\gamma}(Q_r)}&\coloneqq \|u\|_{L^{\infty}(Q_r)}+\|u_t\|_{\mathcal{C}^{0,\gamma}(Q_r)}+\|Du\|_{\mathcal{C}^{1,\gamma}(Q_r)}<+\infty.
	\end{aligned}
\end{equation*}
Hence, every component of the Hessian $D^2u$ is $\gamma$-H\"older continuous with respect to the spatial variables and the derivative of $u$ with respect to the time variable $u_t$ is $\frac{\gamma}{2}$-H\"older continuous in $t$.

Lastly, we say that $u$ belongs to ${\mathcal C}^{1, \llip}_{loc}(Q_r)$ if $u$ satisfies the following estimate
\[
	\displaystyle\sup_{Q_{r/2}(x_0,t_0)}\big|u(x,t)-[u(x_0,t_0)+Du(x_0,t_0)\cdot x]\big|\leq Cr^2\ln r^{-1},
	\]
	for some universal constant $C>0$ .
 We refer to \cite{cra_ko_swI}, \cite{M8}, \cite{imbersil} for more details.

\subsection{Assumptions and main results}

In this subsection, we detail the assumptions and main results of the paper. The first assumption concerns the conditions imposed on the matrix $A_{\alpha,\beta}$.

\begin{Assumption}[\it Ellipticity of $\mathcal{A}_{\alpha,\beta}$]\label{assumption1}\rm
	Let $\alpha\in\mathcal{A}$ and $\beta\in\mathcal{B}$, where $\mathcal{A},\mathcal{B}$ are countable sets. We assume that the matrix $A_{\alpha,\beta}:Q_1\times \mathcal{A}\times\mathcal{B}\to\mathbb{R}^{d^2}$ is $(\lambda,\Lambda)$-elliptic; \emph{i.e.} there are constants $0<\lambda\leq\Lambda$ such that 
	\[
	\lambda I\leq A_{\alpha,\beta}(x,t)\leq \Lambda I
	\]
	for every $(x,t)\in Q_1$.
\end{Assumption}

In the sequel, we introduce some integrability conditions on the source term.

\begin{Assumption}[\it Regularity of the source term]\label{assumptionsourceterm}
	Let $p>d+1$. We suppose $f\in L^p(Q_1).$ 
\end{Assumption}

An important ingredient in the analysis of the Sobolev regularity regards the smallness regime described in the next assumption. 

\begin{Assumption}[\it Smallness regime for regularity in $W^{2,1;p}$]\label{assumptionsobolev}
	We assume that the matrix $\bar{A}_{\beta}:Q_1\times {\cal B}\to\mathbb{R}^{d^2}$  satisfies
	\[
	|A_{\alpha,\beta}(x,t)-\bar{A}_{\beta}(x,t)| \leq \varepsilon_1
	\]
	uniformly in $(x,t)$, $\alpha$ and $\beta$,	where $\varepsilon_1>0$ is a sufficiently small constant that will be determined later. 
\end{Assumption}

It is worth noting that due to the parabolic nature, we need a slightly stronger assumption than in \cite{Pimentel}. Here, we assume that solutions to our approximated problem have ${\mathcal C}^{1,1}$-estimates. The next condition is fundamental to produce Sobolev estimates, since we connect the equation \eqref{equation03} to a Bellman parabolic model.

\begin{Assumption} [\it ${\mathcal C}^{1,1}$- estimates for the parabolic Bellman model] \label{assumption3}
	Let $v\in\mathcal{C}(Q_{8/9})$ be a $L^p$-viscosity solution to
	\[
	v_t+\inf_{\beta \in {\cal B}}[-\tr(\bar{A}_{\beta}(x,t)D^2v)]=0 \ \ \mbox{in} \ \ Q_{8/9}.
	\]
Then $v\in {\mathcal C}^{1,1}(Q_{3/4})\cap\mathcal{C}(\bar{Q}_{3/4})$. Moreover, there exists a universal constant $C>0$ such that
	\[
	\|v\|_{{\mathcal C}^{1,1}(Q_{3/4})}\leq C\|v\|_{L^{\infty}(Q_{8/9})}.
	\]
\end{Assumption}

In order to obtain Sobolev estimates to equation \eqref{equation03}, some assumptions on the lower-order coefficients are required.

\begin{Assumption}[\it The vector $\bf{b}_{\alpha,\beta}$]\label{assumption_vectorb} \rm
	We assume that $\textbf{b}_{\alpha,\beta}\in L^\infty(Q_1)$ uniformly in $\alpha$ and $\beta$; \emph{i.e.}, there exists a constant $C > 0$ such that
	\[
	\sup_{\alpha \in {\cal A}}\sup_{\beta \in {\cal B}}\|\textbf{b}_{\alpha,\beta}\|_{L^{\infty}(Q_1)}\leq C.
	\]
\end{Assumption} 

For the proof of parabolic $\mathcal{C}^{1, \llip}$-estimates, we need to refine the smallness regime on the matrix $A_{\alpha,\beta}$. This is the content of our next assumption.

\begin{Assumption}[\it Smallness regime for regularity in $\mathcal{C}^{1,\llip}$] \label{assumption4}
	We suppose $f \in \bmo(Q_1)$; \emph{i.e.}, for all $Q_r(x_0, t_0) \subset Q_1$, we have 
	\[
	\| f \|_{\bmo(Q_1)}:= \sup_{0<r\leq 1}\dashint_{Q_r(x_0, t_0)} |f(x,t)- \langle f\rangle_{(x_0, t_0), r}| dxdt< \infty,
	\]
	where $\langle f\rangle_{(x_0, t_0), r}:= \displaystyle\dashint_{Q_r(x_0, t_0)} f(x,t) dxdt$. In addition, for every $Q_{r}(x_0, t_0)\subset Q_1,$ we assume that
	\[
	\displaystyle\sup_{Q_{r}(x_0, t_0)}|A_{\alpha,\beta}(x,t)-\bar{A}_{\beta}(x_0,t_0)| \leq\varepsilon_2,
	\]
	uniformly in $\alpha$ and $\beta,$ where $\varepsilon_2>0$ is a sufficiently small constant that will be determined later.
\end{Assumption}

Our last assumption concerns ${\mathcal C}^{2,\gamma}$-estimates at the origin of solutions to \eqref{eq_main} which requires an additional smallness condition.

\begin{Assumption}[\it Smallness regime for ${\mathcal C}^{2,\gamma}$-estimates at the origin]\label{assumption5}
Assume that 
	\[
	\sup_{(x,t)\in Q_{r}}\sup_{\alpha \in {\cal A}}\sup_{\beta \in {\cal B}}|A_{\alpha,\beta}(x,t)-\bar{A}_{\beta}(0,0)|\leq\varepsilon_3 r^{\gamma}.
	\]
In addition, we suppose  
\[
\dashint_{Q_r}|f(x,t)|^pdx \leq \varepsilon_3^pr^{\gamma p},
\]
where $\varepsilon_3>0$ is a sufficiently small constant that will be determined later.
\end{Assumption}

An example of a matrix that satisfies A$\ref{assumption5}$ is given by $A_{\alpha,\beta}(x,t):= \bar{A}_\beta(0,0) + \varepsilon_3|x|^\gamma$.

 It is worth to highlight that, since the assumptions to be made throughout this manuscript are exactly the parabolic counterpart of those imposed in the stationary case, our results are also dependendent on the smallness regime imposed on the coefficients as in \cite{Pimentel}. 

At this point, we put forward our main results. At first, we study the gradient-dependent equation \eqref{equation03}. Under the assumption that the coefficients are uniformly close to the Bellman one, we prove that solutions to \eqref{equation03} are of class $W_{loc}^{2,1;p}(Q_1)$.

\begin{teo}\label{theorem01}
	Let $d +1<p$ and $u \in {\cal C}(Q_1)$ be a $L^p$-viscosity solution to \eqref{equation03}. Assume that A\ref{assumption1}-A\ref{assumption_vectorb} are in force. Then, $u \in W_{loc}^{2,1;p}(Q_1)$ with the estimate
	\[
	\|u\|_{{ W^{2,1;p}}(Q_{1/2})} \leq C \left( \|u\|_{L^{\infty}(Q_1)} +  \|f\|_{L^{p}(Q_1)}\right),
	\]
	where $C>0$ is a constant depending on $d, \lambda, \Lambda, p, \displaystyle\sup_{\alpha \in {\cal A}}\sup_{\beta \in {\cal B}}\|\textbf{b}_{\alpha,\beta}\|_{L^{\infty}(Q_1)}$.	
\end{teo}

We point out that in \cite{Escauriaza}, the author proves $W^{2,p}$ estimates in the elliptic setting for $p> d - \varepsilon$, where $\varepsilon$ is a positive constant and depending on the ellipticity coefficients. This number $\varepsilon$ is well-known in the literature as {\it Escauriaza's exponent}.  According to \cite[Remark I]{Escauriaza}, these results can be produced for the parabolic scenario. However, as far as we know no results in that direction have been produced.  The essential ingredients for the proof of this result can be found in \cite{Cerutti-Grimaldi-07, Chen-17, Escauriaza-2000}; namely, Green's functions properties associated with some linear operators and well-posedness to certain parabolic problems. See also \cite[Section 5]{caspim} and the references therein.

Notice that the matrix $A_{\alpha, \beta}$ depends on $\beta$ in Theorem \ref{theorem01}, it implies that the operator {\it is not} convex; compare with \cite[Theorem 9.1]{cra_ko_swI}. An adjustment in the smallness regime leads us to our second main result that regards the parabolic ${\mathcal C}^{1,\llip}$ regularity.

\begin{teo}\label{theorem02}
	Let $u\in\mathcal{C}(Q_1)$ be a $L^p$-viscosity solution to \eqref{eq_main} and $(x_0,t_0)\in Q_{1/2}$. Suppose A\ref{assumption1} and A\ref{assumption4} are in force. Then $u\in\mathcal{C}^{1,\llip}_{loc}(Q_1)$; \emph{i.e.}, there exist a universal constant $C>0$ and $0<r\leq 1/2$ such that
	\small
	\[
	\displaystyle\sup_{Q_{r}(x_0,t_0)}\big|u(x,t)-[u(x_0,t_0)+Du(x_0,t_0)\cdot x]\big|\leq C\left(\|u\|_{L^{\infty}(Q_1)}+\|f\|_{\bmo(Q_1)}\right)r^2\ln r^{-1}.
	\]
\end{teo}

Lastly, assuming additional conditions on the source term, and refining the smallness regime, we are able to obtain $C^{2,\gamma}$-estimates at the origin for solutions to \eqref{eq_main}. It is the content of our last main theorem.

\begin{teo}\label{theorem03}
	Let $u \in {\mathcal C}(Q_1)$ be a $L^p$-viscosity solution to \eqref{eq_main}. Suppose that assumptions A\ref{assumption1} and A\ref{assumption5} are in force. Then, $u$ is ${\mathcal C}^{2,\gamma}$ at the origin, \emph{i. e.}, there exists a polynomial $P$ of degree 2 and a constant $C>0$ such that 
\[
\|u - P\|_{L^\infty(Q_r)} \leq Cr^{2+\gamma},
\] 
with
\[
|DP(0,0)| + \|D^2P(0,0)\| \leq C,
\]
for all $0<r \ll 1$.
\end{teo}

It is worth noting that in \cite{wangII} the authors establish ${\mathcal C}^{2,\gamma}$ regularity for solutions, relying on ${ \mathcal C}^{2,\gamma}$-estimates for the operator with frozen coefficients, which is not our case.  We also notice that Theorem \ref{theorem03} \emph{does not} implies local ${\mathcal C}^{2,\gamma}$-regularity, unless assumption A$\ref{assumption5}$ holds for every $(x,t)$, see Remark \ref{last_rem}. Throughout the paper, we use some definitions and preliminary results, which are described in the next section.

\begin{Remark}
The authors believe that these results can be extended to operators with zero-th order terms of the form
\[
G(D^2u, u_t, u, x, t) := u_t +  \sup_{\alpha \in {\cal A}} \inf_{\beta \in {\cal B}}\left[ - {\tr}(A_{\alpha, \beta}(x,t)D^{2} u) + a_{\alpha,\beta}(x,t)u(x,t)\right],
\]
by imposing that
\[
	\sup_{\alpha \in {\cal A}}\sup_{\beta \in {\cal B}}\|a_{\alpha,\beta}\|_{L^{\infty}(Q_1)}\leq C.
\]
\end{Remark}

\begin{Remark} 
In the proof of Theorem \ref{theorem01}, we consider the following smallness regimes
\begin{equation}\label{eq_scal}
\|u\|_{L^\infty(Q_1)} \leq 1 \;\;\mbox{ and } \;\; \|f\|_{L^p(Q_1)} \leq \varepsilon_1,
\end{equation} 
for some $\varepsilon_1$ to be determined. The conditions in \eqref{eq_scal} are not restrictive. Indeed, if we consider the auxiliary function
\[
v(x,t) = \dfrac{u(\rho x, \rho^2t)}{K},
\]
with $0 < \rho \ll 1$ and $K > 0$, then $v$ solves 
\[
v_t +  \sup_{\alpha \in {\cal A}} \inf_{\beta \in {\cal B}}\left[ - {\tr}(\tilde{A}_{\alpha, \beta}(x,t) D^{2} v)-\tilde{{\bf b}}_{\alpha,\beta}(x,t)\cdot Dv\right] = \tilde{f} \quad \text{in} \quad Q_1,
\]
in the viscosity sense,  where 
\[
\tilde{A}_{\alpha, \beta}(x,t)  =  A_{\alpha, \beta}(\rho x,\rho^2t),\quad \tilde{{\bf b}}_{\alpha,\beta}(x,t) = \rho{\bf b}_{\alpha,\beta}(\rho x,\rho^2t),
\]
and
\[
\tilde{f}(x,t) = \frac{\rho^2}{K}f(\rho x, \rho^2t).
\]  
Thus, by choosing 
\[
K = \|u\|_{L^\infty(Q_1)} + \varepsilon_1^{-1}\|f\|_{L^p(Q_1)},
\]
we can assume \eqref{eq_scal} without loss of generality, since the coefficients $\tilde{A}_{\alpha, \beta}$ and $\tilde{{\bf b}}_{\alpha,\beta}$ and the source term $\tilde{f}$ satisfy the same assumptions required in Theorem \ref{theorem01}. Similarly, we can assume 
\begin{equation*}
\|f\|_{\bmo(Q_1)} \leq \varepsilon_2,
\end{equation*} 
in the proof of Theorem \ref{theorem02}.
\end{Remark}

\subsection{Definitions and auxiliary results}

In what follows, we recall some definitions and results which will be useful throughout the paper. First, we present the definition of $L^p$-viscosity solution.  

\begin{Definition}[$L^p$-viscosity solution]\label{def Lp-viscosity sol}
Let $f\in L^p_{\textrm{loc}}(Q_1)$. We say that $u\in \mathcal{C}(Q_1)$ is an $L^p$-viscosity subsolution $($resp. supersolution$)$ of 
\begin{equation}\label{Lp-viscosity sol}
u_t+F(x,t,u,Du,D^2u)=f(x,t) \: \: \mbox{in} \: \: Q_1,
\end{equation}
if for $\phi\in  W^{2,p}_{\mathrm{loc}}(Q_1)$, we have
\begin{align*}\label{limSubsolution1}
{\mathrm{ess.}\varliminf}_{(y,s)\to (x,t)} \,\{\phi_t(y,s)+F(y,s,u(y),D\phi(y),D^2\phi (y))-f(y,s)\} \leq 0
\\
(\mbox{resp.},\:{\mathrm{ess.}\varlimsup}_{(y,s)\to (x,t)} \,\{\phi_t(y,s)+F(y,s,u(y),D\phi(y),D^2\phi (y))-f(y,s)\} \geq 0 ) \nonumber
\end{align*}
whenever $u-\phi$ attains a local maximum (resp.\ minimum) at $(x,t) \in Q_1$. We call $u$ is an $L^p$-viscosity solution of \eqref{Lp-viscosity sol}, if  $u$ is an $L^p$-viscosity subsolution and supersolution of \eqref{Lp-viscosity sol}. We say that a $L^p$-viscosity solution $u$ is a normalized $L^p$-viscosity solution if $\sup_{Q_1}|u| \leq 1$.
\end{Definition}

For the sake of completeness, we define the class of viscosity solutions. First, we recall the definition of extremal operators; see \cite{Astesiano} for a first contribution on fully nonlinear parabolic extremal equations.

\begin{Definition}[\it Pucci's extremal operators]\rm
	Let $\mathcal{S}(d)$ the space of $d\times d$ symmetric matrices. For $M \in \mathcal{S}(d)$, we define the Pucci's extremal operators by
	\[
	\mathcal{M}^+_{\lambda,\Lambda}(M)\,:=\,-\lambda\sum_{e_i>0}e_i\,-\,\Lambda\sum_{e_i<0}e_i
	\]
	and
	\[
	\mathcal{M}^-_{\lambda,\Lambda}(M)\,:=\,-\Lambda\sum_{e_i>0}e_i\,-\,\lambda\sum_{e_i<0}e_i,
	\]
	where $(e_i)_{i=1}^d$ are the eigenvalues of $M$.
\end{Definition}

Observe that, if $A\in\mathcal{S}(d)$ is a $(\lambda,\Lambda)$-elliptic matrix, \emph{i.e}, 
\begin{equation}\label{elliptic-pucci}
\lambda |\xi|^2\leq A_{ij}\xi_i\xi_j\leq \Lambda |\xi|^2
\end{equation}
for every $\xi\in\mathbb{R}^d$, it easy to see that we can write the Pucci's extremal operators as
\begin{equation}
	\mathcal{M}^+_{\lambda,\Lambda}(M)=\sup_{\lambda I\leq A\leq \Lambda I}[-\tr(AM)]
\end{equation}
and
\begin{equation}
	\mathcal{M}^-_{\lambda,\Lambda}(M)=\inf_{\lambda I\leq A\leq \Lambda I}[-\tr(AM)];
\end{equation}
we refer to the reader to \cite{cafcab} for more details. See also \cite{cra_ko_swI}, \cite{imbersil}. Therefore, the Pucci's extremal operators are prototype examples of Bellman operators. 

\begin{Definition}[\it The class of viscosity solutions]\rm
	Let $f\in\mathcal{C}(Q_1)$ and $0<\lambda\leq\Lambda.$ We say that $u$ is in the class of supersolutions $\overline{S}(\lambda, \Lambda, f)$ if 
	\[
	u_t+\mathcal{M}^+_{\lambda,\Lambda}(D^2u)\,\geq\, f(x,t) \,\, \mbox{in} \,\, Q_1
	\]
	in the viscosity sense. Similarly, $u$ is in the class of subsolutions $\underline{S}(\lambda, \Lambda, f)$ if
	\[
	u_t+\mathcal{M}^-_{\lambda,\Lambda}(D^2u)\,\leq\, f(x,t) \,\, \mbox{in} \,\, Q_1
	\]
	in the viscosity sense. Finally, the class of $(\lambda, \Lambda)$-viscosity solutions is defined by
	\[
	S(\lambda,\Lambda,f)=\overline{S}(\lambda, \Lambda, f)\cap\underline{S}(\lambda, \Lambda, f).
	\]
\end{Definition}

In what follows we introduce measure notions that we use in the next section. We refer the reader to \cite{caffarelli} for more details.

\begin{Definition}\rm
	Let $L:Q_1\rightarrow\mathbb{R}$ be an affine function and $M$ a positive constant. The paraboloid of opening $M$ is defined by
	\[
	P_{M}(x,t)=L(x,t)\pm M(|x|^2+|t|).
	\]
In addition, we introduce	
	\[
	\underline{G}_{M}(u, Q)\coloneqq\{(x_0,t_0)\in Q: \exists \, P_{M} \; \mbox{that touches } u \mbox{ by bellow at} \; (x_0,t_0)\},
	\] 
	\[
	\overline{G}_{M}(u, Q)\coloneqq\{(x_0,t_0)\in Q: \exists \, P_{M} \; \mbox{that touches } u \mbox{ by above at} \; (x_0,t_0)\},
	\] 
	and
	\[
	G_M(u,Q)\coloneqq \underline{G}_{M}(u, Q)\cap \overline{G}_{M}(u, Q).
	\]
	In addition, denote
	\[
	\underline{A}_M(u,Q)\coloneqq Q\setminus\underline{G}_M(u,Q),\,\,\,\,\,\; \overline{A}_M(u,Q)\coloneqq Q\setminus\overline{G}_M(u,Q)
	\]
	and
	\[
	A_M(u,Q)\coloneqq \underline{A}_{M}(u, Q)\cup \overline{A}_{M}(u, Q).
	\]
\end{Definition}

We close this section with a well-known result from the realm of measure theory. Given $K_1$, a dyadic cube is obtained by repeating a finite numbers of time the following procedure: We split the sides of $K_1$ into two equal intervals in $x$ and four equals one in $t$. We do the same with the $2^{d+2}$ cubes obtained, and we repeat this process. Each cube obtained in this process is called a dyadic cube. We say that $\tilde{K}$ is a predecessor of a cube $K$ if K is one of the $2^{d+2}$ cubes obtained by splitting the sides of $\tilde{K}$.

In addition, given $m \in \mathbb{N}$ and a dyadic cube $K$, the set $\bar{K}^m$ is obtained by staking $m$ copies of its predecessor $\bar{K}$; in other words, if $\bar{K}$ has the form $(a,b)\times L$, then $\bar{K}^m = (b,b + m(b-a))\times L$.


\begin{Lemma}[Stacked covering lemma]\label{lem_cov}
Let $m \in \mathbb{N}$, $A \subset B \subset K_1$ and $0<\rho <1$. Suppose that
\begin{itemize}
\item[(i)] $|A| \leq \rho|K_1|$;
\item[(ii)] If $K$ is dyadic cube of $K_1$ such that $|K \cap A| > \rho|K|$, then $\bar{K}^m\subset B$.
\end{itemize}
Then $|A| \leq \dfrac{\rho(m+1)}{m} |B|$.
\end{Lemma}

For a proof of Lemma \ref{lem_cov} we refer \cite[Lemma 4.27]{imbersil}; see also \cite{cafcab}. The next section is devoted to prove the Theorem \ref{theorem01}.

\section{Estimates in Sobolev spaces}

Throughout this section, we detail the proof of Theorem \ref{theorem01}, namely, the $W^{2,1;p}$-estimates to equation \eqref{equation03}. First, we establish the same estimate for \eqref{eq_main}, \emph{i.e.}, the PDE with no dependence on the gradient. 

\begin{Proposition}\label{prop_sob}
Let $d +1<p$ and $u \in {\cal C}(Q_1)$ be a normalized $L^p$-viscosity solution to \eqref{eq_main}. Assume A\ref{assumption1}-A\ref{assumption3} hold true. Then, $u \in W_{loc}^{2,1;p}(Q_1)$. Moreover, there exists a universal constant $C>0$ such that 
\[
\|u\|_{{ W^{2,1;p}}(Q_{1/2})} \leq C \left( \|u\|_{L^{\infty}(Q_1)} + 
\|f\|_{L^{p}(Q_1)}\right).
\]	
\end{Proposition} 

We use standard arguments to prove Proposition \ref{prop_sob}, see for instance \cite{cafcab} and \cite{caspim}, just to cite a few. On account of completeness, we present the main steps of the proof. We start with the following lemma.

\begin{Lemma}[A priori regularity in $W_{loc}^{2, 1; \delta}(Q_1)$] \label{Lemma01}
Let $u\in{\cal C}(Q_1)$ be a normalized viscosity solution to (\ref{eq_main}). Assume A\ref{assumption1}-A\ref{assumptionsourceterm} are in force. Then, there exist some $\delta>0$ and a universal constant $C>0$ satisfying
\[ 
| A_M(u, Q_1)\cap K_1| \leq C M^{-\delta}.
\]
\end{Lemma}

The Lemma \ref{Lemma01} is a well-known result; we refer the reader to \cite[Proposition 7.4]{cafcab} for the elliptic setting. For the parabolic context, it follows from \cite[Theorem 4.11]{wangI}. Next, we prove an approximation lemma that relates solutions to \eqref{eq_main} with solutions of the Bellman parabolic model.

\begin{Proposition}[First approximation lemma]\label{approximationlemma}
Let $u\in\mathcal{C}(Q_1)$ be a normalized $L^p$-viscosity solution to \eqref{eq_main}. 
Suppose that A\ref{assumption1}-A\ref{assumption3} hold true. Then, given $\delta>0$, there exists $\varepsilon_1>0$, such that, if 
\[
\|f \|_{L^{p}(Q_1)} \leq \varepsilon_1, 
\]
then there exists $h \in {\mathcal C}^{1,1}(Q_{3/4})$ satisfying
$$ \| u - h\|_{L^{\infty}(Q_{3/4})} \leq \delta.$$
\end{Proposition}

\begin{proof}
Suppose the statement of the proposition is false. Then, there exists a $\delta_0 > 0$ such that
$$\|u-h\|_{L^{\infty}(Q_{3/4})} > \delta_0,$$
for every $h \in {\mathcal C}^{1,1}(Q_{3/4})$. Consider the sequences $(A_{\alpha, \beta}^n)_{n\in\mathbb{N}}, (f_n)_{n\in\mathbb{N}}$ and $(u_n)_{n\in\mathbb{N}}$ such that
$$|A^n_{\alpha, \beta}(x,t) - \bar{A}_{\beta}(x,t)| + \|f_n\|_{L^p(Q_{3/4})} \leq 1/n,$$
and $u_n$ solves
\begin{equation}\label{eq02}
(u_n)_t +  \sup_{\alpha \in {\cal A}} \inf_{\beta \in {\cal B}}
[ - {\tr}(A^n_{\alpha, \beta}(x, t) D^{2} u_n(x, t))] = f_n(x, t) \quad \text{in} \quad Q_1.
\end{equation}
The regularity theory available for (\ref{eq02}) implies that, through a subsequence if necessary, $u_n$ converges to a function $u_{\infty}$ in the ${\mathcal C}^{0,\gamma}$-topology; see \cite{imbersil}, \cite{KrySaf}. Now, by standard stability results of viscosity solutions, we have that
\[
(u_{\infty})_t + \inf_{\beta \in {\cal B}} [ - {\tr}(\bar{A}_{\beta}(x, t) D^{2} u_{\infty})] = 0;
\]
see \cite{cra_ko_swI}, \cite{imbersil}.
From assumption A\ref{assumption3} we have $u_{\infty} \in {\mathcal C}^{1,1}(Q_{3/4})$. Finally, taking $h=u_{\infty}$ we obtain a contradiction. This finishes the proof.
\end{proof}

Now, we are able to establish a first level of improved decay rate. In the sequel, $Q$ is a parabolic domain such that $Q_{8\sqrt{d}} \subset Q$.

\begin{Proposition} \label{Prop 1}
Let $0<\rho<1$ and $u\in\mathcal{C}(Q)$ be a normalized $L^p$-viscosity solution to \eqref{eq_main} in $ Q_{8{\sqrt{d}}}$ satisfying
\[
- | x |^2 - | t | \leq u(x,t) \leq |x |^2 + | t | \quad \mbox{in} \quad Q \setminus Q_{6{\sqrt{d}}}.
\]
Assume that A\ref{assumption1}-A\ref{assumption3} are satisfied and also 
\[
\| f \|_{L^{d+1}(Q_{8{\sqrt{d}}})} \leq \varepsilon.
\]
Then, there exists $\bar{M}>1$ such that  
\[
| G_{\bar{M}}(u, Q) \cap K_1| \geq 1 - \rho. 
\]
\end{Proposition}

\begin{proof}
From Proposition \ref{approximationlemma}, there exists $h\in {\mathcal C}^{1,1}_{loc}(Q_{8\sqrt{d}})$ such that
	\[
		\|u-h\|_{L^{\infty}(Q_{6{\sqrt{d}}})}\leq\delta.
	\]
Extend $h$ continuously to $Q$ such that 
\[ 
h = u \quad  \text{in} \quad Q \setminus Q_{7{\sqrt{d}}}$$ and $$ \| u - h\|_{L^{\infty} (Q)} = \| u - h\|_{L^{\infty} (Q_{6{\sqrt{d}}})}.
\]
By the maximum principle we obtain 
\[
\| u \|_{L^{\infty} (Q_{6{\sqrt{d}}})} = \|  h\|_{L^{\infty} (Q_{6{\sqrt{d}}})}.
\]
It follows that
\[
\| u - h\|_{L^{\infty} (Q)} \leq 2 
\]
and
\[
-2 - | x |^2 - | t | \leq h(x,t) \leq 2  + |x |^2 + | t | \quad  \text{in} \quad Q \setminus Q_{6{\sqrt{d}}}.
\]
Hence, we can find $N>1$ such that $Q_1 \subset G_{N}(h, Q)$.

Now, we introduce the auxiliary function
\[
 w := \frac{\delta}{2C \varepsilon}(u - h).
 \]
According to Lemma \ref{Lemma01} applied to $ w \in S(\lambda, \Lambda, f)$ we have
\[
|A_{M_1}(w, Q)\cap K_1 | \leq C M_1^{-\sigma},
\]
for every $M_1>0$, which leads to
\[
|A_{M_2}(u - h , Q)\cap K_1| \leq C {\varepsilon}^{\sigma}M_2^{\sigma} 
\]
for every $M_2>0$.
Therefore
\[
|G_{N}(u - h, Q) \cap K_1| \geq 1 - C{\varepsilon}^{\sigma}M_2.
\]
By choosing $\varepsilon \ll 1$ sufficiently small, and taking $\bar{M} \equiv 2N$, we conclude the proof.
\end{proof}

\begin{Proposition} \label{Prop 2}
Let $0<\rho<1$ and $u\in\mathcal{C}(Q)$ be a normalized $L^p$-viscosity solution to
\eqref{eq_main} in $Q_{8\sqrt{d}}$. Assume A\ref{assumption1}-A\ref{assumption3} are in force. In addition, suppose  
\[
\|f\|_{L^{d+1}(Q_{8\sqrt{d}})} \leq \varepsilon,
\]
and $G_1(u,Q) \cap K_3 \not= \emptyset.$ Then 
\[
|G_M(u, Q) \cap K_1| \geq 1 - \rho,
\]
with $M$ as in Proposition \ref{Prop 1}. 
\end{Proposition}

\begin{proof}
Let $(x_1, t_1) \in G_1(u,Q)\cap K_3$. It implies that there exists an affine function $L$ such that
\[-\dfrac{|x-x_1|^2 + |t-t_1|}{2} \leq u(x,t) - L(x,t) \leq \dfrac{|x-x_1|^2 + |t-t_1|}{2} \;\;\; \text{in}\;\; Q.
\]
Now, we set 
\[
v:= \dfrac{u-L}{C},
\]
where $C>1$ is a large constant such that $\|v\|_{L^{\infty}(Q_{8\sqrt{d}})} \leq 1$ and
\[
-|x|^2 - |t| \leq v(x,t) \leq |x|^2 + |t| \;\ \text{in} \;\ Q\setminus Q_{6\sqrt{d}}.
\]
Notice that $v$ solves
\[
v_t + \sup_{\alpha \in {\cal A}}\inf_{\beta \in {\cal B}}(-\tr(A_{\alpha, \beta}(x,t)D^2v)) = \dfrac{f}{C}.
\]
If we set $M:=C\bar{M}$, from Proposition \ref{Prop 1} we obtain
\[
|G_M(u,Q)\cap K_1| = |G_{C\bar{M}}(u,Q)\cap K_1| = |G_{\bar{M}}(v,Q)\cap K_1| \geq 1 - \rho.
\]
\end{proof}

The following result is an application of Lemma \ref{lem_cov} and produces decay rates for the sets $A_M\cap K_1$. 

\begin{Proposition}\label{Prop 4}
Let $0<\rho<1$ and $u\in\mathcal{C}(Q)$ be a normalized $L^p$-viscosity solution to \eqref{eq_main} in $Q_{8\sqrt{d}}$. Extend $f$ by zero outside of $Q_{8\sqrt{d}}$. Suppose A\ref{assumption1}-A\ref{assumption3} hold true. Denote
\[
A\coloneqq A_{M^{k+1}}(u, Q_{8\sqrt{d}})\cap K_1
\]
and
\[
B := \left\lbrace A_{M^{k}}(u, Q_{8\sqrt{d}})\cap K_1\right\rbrace\cup\left\lbrace(x,t)\in K_1: m(f^{d+1})(x,t)\geq(c_1M^{k})^{d+1}\right\rbrace,
\]
where $c_1$ is a positive universal constant and $M>1$ depends only on $d$. Then,
\[
|A|\leq \rho|B|.
\]
\end{Proposition}
\begin{proof}
First, observe that
\[
|u(x,t)|\leq 1\leq |x|^2+|t| \ \ \mbox{in} \ \ Q_{8\sqrt{d}}\backslash Q_{6\sqrt{d}}.
\]
According to Proposition \ref{Prop 1}, we obtain 
\[
|G_{M^{k+1}}(u,Q_{8\sqrt{d}})\cap K_1|\geq 1-\rho,
\]
which implies that
\[
|A|=|A_{M^{k+1}}(u,Q_{8\sqrt{d}})\cap K_1|\leq \rho|K_1|.
\]
Now, consider any dyadic cube $K\coloneqq K_{1/2^i}$ of $K_1.$ Notice that
\begin{equation}\label{hypothesis}
|A_{M^{k+1}}(u,Q_{8\sqrt{d}})\cap K|=|A\cap K|>\rho|K|.
\end{equation}
It remains to see that $\bar{K}^m \subset B$, for some $m \in \mathbb{N}$. We proceed by a contradiction argument assuming that $\bar{K}^m\not\subset B$. Let $(x_1,t_1)$ such that
\begin{equation}\label{inter contradiction}
(x_1,t_1)\in\bar{K}^m\cap G_{M^k}(u, Q_{8\sqrt{d}})
\end{equation}
and
\begin{equation}\label{max_cont}
m(f^{d+1})(x_1,t_1)\leq (c_1M^k)^{d+1}.
\end{equation}
Define
\[
v(x,t)\coloneqq\frac{2^{2i}}{M^k}u\left(\frac{x}{2^i},\frac{t}{2^{2i}}\right).
\]
Since $Q_{8\sqrt{d}} \subset Q_{2^i\cdot8\sqrt{d}}$, we have that $v$ solves
\[
v_t + \sup_{\alpha \in {\cal A}} \inf_{\beta \in {\cal B}}
[ - {\tr}(A_{\alpha, \beta}(x, t) D^{2} v)]=\tilde{f} \ \ \mbox{in} \ \ Q_{8\sqrt{d}},
\]
where 
\[
\tilde{f}(x,t)\coloneqq\frac{1}{M^k}f\left(\frac{x}{2^i},\frac{t}{2^{2i}}\right).
\]
We have
\[
\|\tilde{f}\|_{L^{d+1}(Q_{8\sqrt{d}})}^{d+1}=\frac{2^{i(d+2)}}{M^{k(d+1)}}\displaystyle\int_{Q_{8\sqrt{d}/2^i}}|f(x,t)|^{d+1}dxdt\leq c(d)c_1^{d+1}.
\] 
Now, by choosing $c_1$ small enough in \eqref{max_cont} we obtain
\[
\|\tilde{f}\|_{L^{d+1}(Q_{8\sqrt{d}})}\leq\varepsilon.
\]
Furthermore, the inequality \eqref{inter contradiction} yields 
\[
G_1(v,Q_{8\sqrt{d}/2^i})\cap K_3\neq\emptyset.
\]
From Proposition \ref{Prop 2} we get
\[
|G_{M}(v, Q_{2^i\cdot8\sqrt{d}})\cap K_1|\geq(1-\rho) 
\]
\emph{i.e.},
\[
|G_{M^{k+1}}(u, Q_{8\sqrt{d}})\cap K|\geq(1-\rho)|K|,
\]
which contradicts \eqref{hypothesis}.
\end{proof}

At this point we are ready to prove the Proposition \ref{prop_sob}. 

\begin{proof}[Proof of Proposition \ref{prop_sob}]

Define
\[
\alpha_k\coloneqq|A_{M^k}(u,Q_{8\sqrt{d}})\cap K_1|
\]
and
\[
\beta_k\coloneqq|\{(x,t)\in K_1:m(f^{d+1})(x,t)\geq(c_1M^k)^{d+1}\}|.
\]

From Proposition \ref{Prop 4}, we have that
\[
\alpha_{k+1}\leq\rho(\alpha_k+\beta_k).
\]
Hence,
\begin{equation}\label{inequality01}
\alpha_k\leq\rho^k+	\displaystyle\sum_{i=1}^{k-1}\rho^{k-i}\beta_i.
\end{equation}

Since $f\in L^p(Q_1)$, it follows that $m(f^{d+1})\in L^{p/(d+1)}(Q_1),$ and, for some $C>0$
\[
\|m(f^{d+1})\|_{L^{p/(d+1)}(Q_1)}\leq C\|f\|^{d+1}_{L^p(Q_1)}.
\]
Therefore, 
\begin{equation}\label{inequality02}
\displaystyle\sum_{k=0}^{\infty}M^{pk}\beta_k\leq C.
\end{equation}

By combining \eqref{inequality01} and \eqref{inequality02} and choosing $\rho$ such that $\rho M^p \leq 1/2$, we obtain
\begin{equation*}
\begin{aligned}
\displaystyle\sum_{k=1}^{\infty}M^{pk}\alpha_k&\leq\displaystyle\sum_{k=1}^{\infty}(\rho M^p)^k+\displaystyle\sum_{k=1}^{\infty}\sum_{i=0}^{k-1}\rho^{k-i}M^{p(k-i)}\beta_iM^{pi}\\
&\leq\displaystyle\sum_{k=1}^{\infty}2^{-k}+\left(\displaystyle\sum_{i=0}^{\infty}M^{pi}\beta_i\right)\left(\displaystyle\sum_{j=1}^{\infty}(\rho M^p)^j\right)\\
&\leq \displaystyle\sum_{k=1}^{\infty}2^{-k}+C\displaystyle\sum_{j=1}^{\infty}2^{-j}\\
&\leq C.
\end{aligned}
\end{equation*}
This concludes the proof.
\end{proof}

Finally, we detail the proof of Theorem \ref{theorem01}. 
\begin{proof}[Proof of Theorem \ref{theorem01}]
We split the proof in two steps. \\

\textbf{Step 1}
\bigskip

 First, by a reduction argument, we see that it is enough to prove the result for $L^p$-viscosity solutions to \eqref{eq_main}.

Let $u$ be an $L^p$-viscosity solution to \eqref{equation03}. By \cite[Proposition 3.2]{cra_ko_swI}, $u$ is parabolic twice differentiable a.e. and its pointwise derivatives satisfy \eqref{equation03} in $Q_1$. Define
\[
g(x,t) := u_t +  \sup_{\alpha \in {\cal A}} \inf_{\beta \in {\cal B}} [ - {\tr}(A_{\alpha, \beta}(x,t)D^{2} u)].
\]
It is easy to see that
\[
|g(x,t)| \leq |f(x,t)| + \sup_{\alpha \in {\cal A}} \sup_{\beta \in {\cal B}}|{\bf b}_{\alpha,\beta}(x,t)||Du|.
\]

According to the Theorem 7.3 in \cite{cra_ko_swI}, we have that $Du \in L^p(Q_1)$ with estimates. Furthermore, by Remark 7.7 in \cite{cra_ko_swI}, we obtain
\[
\|Du\|_{L^{p}(Q_{1/2})} \leq C\left(\|u \|_{L^{\infty}(Q_1)} + \| f\|_{L^p(Q_1)}\right), 
\]
for some constant $C>0$. Since ${\bf b}_{\alpha,\beta}$ satisfies A\ref{assumption_vectorb} and $f\in L^p(Q_1)$, we have that $g \in L^p_{loc}({Q_1})$, with $p>d+1$.  It follows from \cite[Proposition 4.1]{cra_ko_swI} that $u$ is an $L^p$-viscosity solution to

\begin{equation}\label{eq_wgrad}
u_t +  \sup_{\alpha \in {\cal A}} \inf_{\beta \in {\cal B}} [ - {\tr}(A_{\alpha, \beta}(x,t)D^{2} u)] = g(x,t) \;\; \text{ in } \;\; Q_1.
\end{equation}
Therefore, if the Theorem \ref{theorem01} holds for $L^p$-viscosity solutions of \eqref{eq_wgrad}, we can conclude the proof.

\bigskip
\textbf{Step 2} 
\bigskip

Now, consider the equation
\begin{equation}\label{eq_wg}
u_t +  \sup_{\alpha \in {\cal A}} \inf_{\beta \in {\cal B}} [ - {\tr}(A_{\alpha, \beta}(x,t)D^{2} u)] = g(x,t) \;\; \text{ in } \;\; Q_1.
\end{equation}
Let $g_j \in {\mathcal C}(\overline{Q}_1)\cap L^p(Q_1)$ and $u_j$ such that
\[
\|g_j - g\|_{L^p(Q_1)} \rightarrow 0, \;\text{ as } \; j\rightarrow \infty,
\]
and
\begin{equation*}
\left\{
\begin{array}{rcl}
(u_j)_t +  \sup_{\alpha \in {\cal A}} \inf_{\beta \in {\cal B}} [ - {\tr}(A_{\alpha, \beta}(x,t)D^{2} u_j)] & =  g_j(x,t) & \text{ in } \;\; Q_1 \\
u_j(x,t) & =  u(x,t) & \mbox{ on } \;\; \partial Q_1. 
\end{array}
\right.
\end{equation*}


By Proposition \ref{prop_sob} we have that 
\[
\|u_j\|_{W^{2,1;p}(Q_{1/2})} \leq C\left( \|u_j\|_{L^\infty(Q_1)} + \|g_j\|_{L^p(Q_1)} \right).
\]
By using \cite[Proposition 2.6]{cra_ko_swI} and Sobolev embeddings (see for instance \cite{LSU}), we obtain that, up to a subsequence if necessary, $u_j \rightarrow \bar{u}$ in ${\mathcal C}(\overline{Q}_1)$. Moreover, $u_j$ converges weakly to $\bar{u}$ in $W_{loc}^{2,1;p}(Q_1)$. By stability results we have that $\bar{u}$ is a $L^{p}$-viscosity solution to \eqref{eq_wg}; see \cite{cra_ko_swI}, \cite{imbersil}. In addition,
\[
\|\bar{u}\|_{W^{2,1;p}(Q_{1/2})} \leq C\left( \|\bar{u}\|_{L^\infty(Q_1)} + \|g\|_{L^p(Q_1)} \right).
\]
Compatibility on the parabolic boundary and the maximum principle \cite[Lemma 6.2]{cra_ko_swI} guarantee that $\bar{u} = u$. This finishes the proof.
\end{proof}

\begin{Remark}
An important development concerning regularity of the solutions in Sobolev spaces to fully nonlinear equations in the elliptic setting was pursuit in \cite{nik}. In that paper, the author develops a global, up to the boundary, estimate in $W^{2,p}$. We believe a similar line of arguments could be developed also in the parabolic setting, leading to a global regularity also in the context of the Isaacs model.
\end{Remark}


\begin{Remark}
We believe that under further conditions on $f$, namely $f \in \bmo$, it would be possible to prove that $D^2u$ and $u_t$ are in $\bmo$, locally. We refer to \cite{tei_pim} for the elliptic case.
\end{Remark}

\section{Regularity in $\mathcal{C}^{1,\llip}$ spaces}

This section is devoted to prove the parabolic $\mathcal{C}^{1,\llip}(Q_1)$ interior regularity estimates for solutions to \eqref{eq_main}. In order to prove this result, initially we establish a second approximation lemma which unlocks the geometric argument. 

\begin{Proposition}[Second approximation lemma]\label{third approximation}
Let $u\in\mathcal{C}(Q_1)$ be an $L^p$-viscosity solution to \eqref{eq_main}. Assume A\ref{assumption1} and A\ref{assumption4} are in force. Given $\delta>0$, there exists $\varepsilon_2>0$ such that, if
\[
\|f \|_{\bmo(Q_1)} \leq \varepsilon_2,
\] 
then we can find $h \in {\cal C}^{2,\bar{\gamma}}(Q_{3/4})$, for some $0<\bar{\gamma}<1$, satisfying 
\begin{equation*}
\left \{
\begin{array}{ll}
\displaystyle h_t + \inf_{\beta \in {\cal B}} [ - {\tr}(\bar{A}_{\beta}(0,0) D^{2}h)] = 0  & \text{in} \quad  Q_{ 3/4}, \\
h = u                                                                                                   & \text{on} \quad  {\partial}Q_{ 3/4}, 
\end{array}
\right. 
\end{equation*}
such that
\[
\|u-h\|_{L^\infty(Q_{3/4})}\leq\delta.
\]
Furthermore, $\|h\|_{\mathcal{C}^{2,\bar{\gamma}}(Q_{3/4})}\leq C$, for some $C>0$ a universal constant.
\end{Proposition}
\begin{proof}
	By a contradiction argument, assume that the statement of the proposition is false. Then, we can find  $\delta_0 > 0$ and sequences $(A_{\alpha, \beta}^n)_{n\in\mathbb{N}}, (f_n)_{n\in\mathbb{N}}$ and $(u_n)_{n\in\mathbb{N}}$ satisfying
	\begin{itemize}
		\item [(i)] $|A^n_{\alpha, \beta}(x,t) - \bar{A}_{\beta}(0,0)| + \|f_n\|_{L^p(Q_{3/4})} \leq 1/n;$
		\item[(ii)] $(u_n)_t +  \displaystyle\sup_{\alpha \in {\cal A}} \inf_{\beta \in {\cal B}}
		\left[ - {\tr}(A^n_{\alpha, \beta}(x, t) D^{2} u_n(x, t))\right]= f_n(x, t)$;
	\end{itemize}
	however
	$$\|u-h\|_{L^{\infty}(Q_{3/4})} > \delta_0,$$
	for every $h \in {\cal C}^{2, \bar\gamma}(Q_{3/4})$ and every $\bar{\gamma}\in(0,1)$.
	
	Because of (ii), the sequence $(u_n)_{n\in\mathbb{N}}$ is uniformly bounded in $\mathcal{C}^{0,\gamma}$, for some $\gamma\in(0,1)$; see \cite{imbersil}, \cite{KrySaf}. Hence $u_n$ converges to a function $u_{\infty}$ locally uniformly in $Q_1$. By standard stability results of viscosity solutions, we have that
	\[
	(u_{\infty})_t + \inf_{\beta \in {\cal B}} [ - {\tr}(\bar{A}_{\beta}(0, 0) D^{2} u_{\infty})] = 0;
	\]
	see \cite{cra_ko_swI}, \cite{imbersil}.
	Since the Bellman operator is convex, the Evans-Krylov's regularity theory assures that $u_{\infty}\in\mathcal{C}^{2, \bar\gamma}(Q_{3/4})$, for some $\bar{\gamma}\in(0,1)$ and that $\|u_{\infty}\|_{\mathcal{C}^{2,\bar\gamma}(Q_{3/4})}\leq C$, with $C>0$ a universal constant; see \cite{krylov1}, \cite{krylov2}. Setting $h=u_{\infty}$ we obtain a contradiction.
\end{proof}

\medskip

The Approximation Lemma provides a tangential path connecting the Bellman parabolic model with our problem of interest. The next Proposition ensures the existence of an approximating quadratic polynomial, which is the key for the proof of $\mathcal{C}^{1,\llip}$-estimates. For simplicity, when the point is the origin,  we denote $\langle f \rangle_r$ instead of $\langle f\rangle_{(0, 0), r}$.

\begin{Proposition}\label{induction1}
Let $u\in\mathcal{C}(Q_1)$ be an $L^p$-viscosity solution to \eqref{eq_main}. Assume A\ref{assumption1} and A\ref{assumption4} hold. Then, there exists $ \varepsilon_2>0$ such that if
\[
\|f \|_{\bmo(Q_1)} \leq \varepsilon_2,
\] 
one can find $0<\rho\ll 1$ and a sequence of second order polynomials $(P_n)_{n\in\mathbb{N}}$ of the form
\[
P_n(x,t)\coloneqq a_n+b_n\cdot x+c_n \,t+\frac{1}{2}x^{t}d_nx
\]
satisfying:
\[
c_n+\displaystyle\inf_{\beta \in {\cal B}}\left[-\tr(\bar{A}_{\beta}(0,0)d_n)\right]=\langle f \rangle_1,
\]
\[
\displaystyle\sup_{Q_{\rho^n}}|u(x,t)-P_n(x,t)|\leq\rho^{2n}
\]
and
\begin{equation}\label{condition coefficients}
|a_{n-1}-a_n|+\rho^{n-1}|b_{n-1}-b_n|+\rho^{2(n-1)}(|c_{n-1}-c_n|+|d_{n-1}-d_n|)\leq C\rho^{2(n-1)},
\end{equation}
for every $n\geq0.$
\end{Proposition}
\begin{proof}
First, we may assume $\|u\|_{L^{\infty}(Q_1)}\leq1/2$, by a reduction argument. We argue by induction in $n \geq0$. We present the proof in four steps. 
	
\medskip
	
\textbf{Step 1}
	
\medskip
	
Define
\[
P_{-1}(x,t)=P_0(x,t)=\frac{1}{2}x^tQx,
\]
where $Q$ is such that
\[
\displaystyle\inf_{\beta \in {\cal B}}[-\tr(\bar{A}_{\beta}(0,0)Q)]=\langle f \rangle_1.
\]
	
The case $n=0$ is obviously satisfied. Suppose the induction hypotheses have been established for $n=1, \dots, k$, for some $k\in\mathbb{N}$. Let us show that the case $n=k+1$ also holds true. Define an auxiliary function $v_k:Q_1\rightarrow\mathbb{R}$ as
\[
v_k(x,t)\coloneqq\frac{(u-P_k)(\rho^kx,\rho^{2k}t)}{\rho^{2k}}.
\]  
	
Observe that $v_k$ solves the equation
\begin{equation*}\label{equation_vk}
(v_k)_{t}+\left(\sup_{\alpha \in {\cal A}} \inf_{\beta \in {\cal B}}
\left[ - {\tr}(A_{\alpha, \beta}(\rho^kx,\rho^{2k}t) (D^{2}v_k+d_k))\right]+c_k\right)=f_k(x,t) \ \ \mbox{in} \ \ Q_1
\end{equation*}
where $f_k(x,t)=f(\rho^kx,\rho^{2k}t).$

\medskip

\textbf{Step 2}

\medskip

By induction hypothesis we conclude that $|v_k|\leq1$. Also, from the assumption A\ref{assumption4}, notice that
\[
|A_{\alpha,\beta}(\rho^kx,\rho^{2k}t)-\bar{A}_{\beta}(0,0)|\leq\varepsilon_2.
\]
Moreover,
\[
\begin{array}{rcl}
\displaystyle \dashint_{Q_r}|f_k(x,t) - \langle f_k\rangle_r|dx dt & = & 
\dfrac{1}{|Q_{r\rho^k}|}\displaystyle \int_{Q_{r\rho^k}}|f(y, s) - \langle f \rangle_{r\rho^k}| dyds \vspace{0.2cm} \\
&\leq &\displaystyle\sup_{0<r\leq1}\displaystyle \dashint_{Q_r}|f(x,t) - \langle f\rangle_r|dx dt\\
& = & \|f\|_{\bmo(Q_1)} \\
& \leq & \varepsilon_2.
\end{array}
\]

Observe that, if we have $v\in\mathcal{C}(Q_1)$ a viscosity solution to 
\[
v_t+\displaystyle\inf_{\beta \in {\cal B}}\left[-\tr(\bar{A}_{\beta}(0,0)D^2v)\right]=0,
\]
then,  by applying the Evans-Krylov's parabolic regularity theory we obtain that $v\in\mathcal{C}^{2,\bar\gamma}_{loc}(Q_1)$, for some $\bar\gamma\in(0,1)$; see \cite{krylov1}, \cite{krylov2}. Furthermore, the following estimate holds 
\[
\|v\|_{\mathcal{C}^{2,\bar\gamma}(Q_{1/2})}\leq C_1,
\]
with $C_1>0$ a universal constant. 

However, from the induction hypothesis, we have
\[
c_k+\displaystyle\inf_{\beta \in {\cal B}}\left[-\tr(\bar{A}_{\beta}(0,0)d_k)\right]=\langle f\rangle_1.
\]
Therefore, it follows that solutions to
\[
v_t+c_k+\displaystyle\inf_{\beta \in {\cal B}}\left[-\tr(\bar{A}_{\beta}(0,0)(D^2v+d_k))\right]=\langle f\rangle_1
\]
are of class $\mathcal{C}^{2,\bar\gamma}_{loc}(Q_1)$, for some $\bar\gamma\in(0,1)$, with estimate
\[
\|v\|_{\mathcal{C}^{2,\bar\gamma}(Q_{1/2})}\leq C=C(\langle f\rangle_1, C_1).
\]
Indeed, if we define the operator 
	\[
	G(M)\coloneqq\displaystyle\inf_{\beta \in {\cal B}}\left[-\tr(\bar{A}_{\beta}(0,0)(M+d_k))\right]-\displaystyle\inf_{\beta \in {\cal B}}\left[-\tr(\bar{A}_{\beta}(0,0)d_k)\right],
	\]
	we obtain that $v$ solves 
	\begin{equation*}
	\begin{aligned}
	v_t+G(D^2v)=&\;v_t+c_k+\displaystyle\inf_{\beta \in {\cal B}}\left[-\tr(\bar{A}_{\beta}(0,0)(D^2v+d_k))\right]\\
	&-c_k-\displaystyle\inf_{\beta \in {\cal B}}\left[-\tr(\bar{A}_{\beta}(0,0)d_k)\right]\\
	=&\;\langle f\rangle_1-\langle f\rangle_1\\
	=&\;0.
	\end{aligned}
	\end{equation*}
Moreover, since the Bellman operator is uniformly elliptic and convex, it follows that $G:\mathcal{S}(d)\rightarrow\mathbb{R}$ is also a uniformly elliptic and convex operator.

\medskip

\textbf{Step 3}

\medskip

As a consequence of Step 2, we have that Proposition \ref{third approximation} holds true for $v_k$. Hence, we can find a function $h\in\mathcal{C}^{2,\bar\gamma}_{loc}(Q_1)$, for some $\bar\gamma\in(0,1)$, satisfying
\begin{equation}\label{equationh1}
h_t+c_k+\displaystyle\inf_{\beta \in {\cal B}}[-\tr(\bar{A}_{\beta}(0,0)(D^2h+d_k))]=\langle f\rangle_1 \ \ \mbox{in} \ \ Q_1,
\end{equation}
such that 
\[
\sup_{Q_\rho}|v_k(x,t)-h(x,t)|\leq\delta
\]
for given $\delta>0$ which we choose below.

Define
\[
\bar{P}(x,t)\coloneqq h(0,0)+Dh(0,0)\cdot x+h_t(0,0)\,t+\frac{1}{2}x^t D^2h(0,0)x.
\]
Then, since $h\in\mathcal{C}^{2,\bar\gamma}_{loc}(Q_1),$ we have that
\begin{equation}\label{estimate h}
|D^2h(0,0)|+|h_t(0,0)|+|Dh(0,0)|+|h(0,0)|\leq C
\end{equation}
and
\[
\sup_{Q_\rho}|h(x,t)-\bar{P}(x,t)|\leq C\rho^{2+\bar{\gamma}}.
\]
Therefore, from the triangular inequality, it follows that
\begin{equation*}\label{triangular ine}
\begin{aligned}
\displaystyle\sup_{Q_\rho}|v_k(x,t)-\bar{P}(x,t)|&\leq \displaystyle\sup_{Q_\rho}|v_k(x,t)-h(x,t)|+\displaystyle\sup_{Q_\rho}|h(x,t)-\bar{P}(x,t)|\\
&\leq\delta+C\rho^{2+\bar{\gamma}}.
\end{aligned}
\end{equation*}
In the sequel, we make the universal choices
\[
\delta\coloneqq\frac{\rho^2}{2} \ \ \ \mbox{and} \ \ \ \rho\coloneqq\left(\frac{1}{2C}\right)^{1/\bar{\gamma}}
\]
to obtain
\begin{equation}\label{k+1-step1}
\displaystyle\sup_{Q_\rho}|v_k(x,t)-\bar{P}(x,t)|\leq\rho^2.
\end{equation}

Setting
\[
P_{k+1}(x,t)\coloneqq P_k(x,t)+\rho^{2k}\bar{P}(\rho^{-k}x,\rho^{-2k}t)
\]
we can conclude from \eqref{k+1-step1}
\[
\displaystyle\sup_{Q_{\rho^{k+1}}}|u(x,t)-P_{k+1}(x,t)|\leq\rho^{2(k+1)}.
\]

Note that by choosing $\rho$, we fix $\delta$ which determines the value of $\varepsilon_2$.

\medskip

\textbf{Step 4}

\medskip

From the definition of $P_{k+1}$ we have that $c_{k+1}=c_k+h_t(0,0)$ and $d_{k+1}=d_k+D^2h(0,0);$ therefore from \eqref{equationh1}, we have
\[
c_{k+1}+\displaystyle\inf_{\beta \in {\cal B}}[-\tr(\bar{A}_{\beta}(0,0)d_{k+1})]=\langle f\rangle_1.
\]

To conclude the $(k+1)$-th step of the induction, note that, since $a_{k+1}=a_k+\rho^{2k}h(0,0)$ and $b_{k+1}=b_k+\rho^k Dh(0,0),$ from \eqref{estimate h} we obtain that
\[
|a_{k+1}-a_k|+\rho^{k}|b_{k+1}-b_k|+\rho^{2k}(|c_{k+1}-c_k|+|d_{k+1}-d_k|)\leq C\rho^{2k}.
\]
The proof of the proposition is now complete.
\end{proof}

\medskip

Finally, we are able to prove the Theorem \ref{theorem02} which we describe in details below.

\begin{proof}[Proof of the Theorem \ref{theorem02}]
Without loss of generality, consider $(x_0,t_0)=(0,0).$ First, it follows from \eqref{condition coefficients} that the sequences $(a_n)_{n\in\mathbb{N}}$ and $(b_n)_{n\in\mathbb{N}}$ are convergent sequences to $u(0,0)$ and $Du(0,0)$, respectively. Moreover,
\[
|a_n-u(0,0)|\leq C\rho^{2n} \ \ \mbox{and} \ \ |b_n-Du(0,0)|\leq C\rho^n.
\]
	
Furthermore, the estimate in \eqref{condition coefficients} yields
\[
|c_n|\leq\displaystyle\sum_{j=1}^{n}|c_j-c_{j-1}|\leq Cn
\]
and
\[
|d_n|\leq\displaystyle\sum_{j=1}^{n}|d_j-d_{j-1}|\leq Cn.
\]
	
Let $0<r\ll 1$ and fix $n\in\mathbb{N}$ such that $\rho^{n+1}<r<\rho^n$. Hence, we estimate from the previous computations
\begin{equation*}
\begin{array}{rl}
&\displaystyle\sup_{Q_{r}}\left|u(x,t)-[u(0,0)+Du(0,0)\cdot x]\right| \vspace{0.2cm}\\
&\quad\quad\leq \displaystyle\sup_{Q_{\rho^n}}\left|u(x,t)-P_n(x,t)\right|+\displaystyle\sup_{Q_{\rho^n}}\left|P_n(x,t)-[u(0,0)+Du(0,0)\cdot x]\right|\\
&\quad\quad\leq  \rho^{2n}+|a_n-u(0,0)| + \displaystyle \rho^n|b_n-Du(0,0)| + \rho^{2n}(|c_n|+|d_n|)\vspace{0.2cm}\\
&\quad\quad\leq C\rho^{2n} + C\rho^{2n} + \rho^{2n}|c_n| + \rho^{2n}|d_n|\vspace{0.2cm}\\
&\quad\quad\leq   Cn \rho^{2n} \vspace{0.2cm}\\
&\quad\quad\leq   -\frac{2C}{\rho^2\ln\rho} r^2\ln r^{-1},

\end{array}
\end{equation*}
The last inequality follows from the fact that $r <\rho^n$ implies $n< \frac{\mbox{ln} \;r}{\mbox{ln} \;\rho}$.  This finishes the proof.
\end{proof}

\begin{Remark}
In general, one of the important facts concerning the study of ${\mathcal C}^{\llip}$ regularity lies on the following: assume $u \in {\mathcal C}^{\llip}(Q)$, $Q\Subset Q_1$ and let $\gamma \in (0,1)$. Notice that
\[
\lim_{s\to 0^+}-s^{1-\gamma}\ln s = 0.
\]
Thus, 
\[
s^{1-\gamma}\ln(1/s) \leq C = C(\gamma), \;\;\mbox{ for } \;\; s <1/2.
\]
It implies that
\begin{equation*}
	\begin{aligned}
		|u(x_1, t_1) - u(x_2,t_2)| &\leq C \dist((x_1,t_1),(x_2,t_2))\ln\left(\frac{1}{\dist((x_1,t_1),(x_2,t_2))}\right)\\
		&\leq C(\gamma)\dist((x_1,t_1),(x_2,t_2))^\gamma.
	\end{aligned}
\end{equation*}
Hence $u \in {\mathcal C}^{0,\gamma}(Q)$. Therefore, once regularity in ${\mathcal C}^{\llip}$ is available, it is possible to conclude that $u \in {\mathcal C}^{0,\gamma}(Q)$, for every $\gamma \in (0,1)$. However, functions in ${\mathcal C}^{\llip}$ spaces may not be Lipschitz continuous. In fact, the Lipschitz logarithmical modulus of continuity $\omega(s)\coloneqq s\ln(1/s)$ is not a Lipschitz continuous function. 
\end{Remark}


\section{Improved ${\mathcal C}^{2,\gamma}$-estimates at the origin}

In this last section, we state and prove $\mathcal{C}^{2,\gamma}$-estimates at the origin. As in the previous section, this is achieved through approximation methods.  In the next result we prove the existence of a sequence of second order polynomials which approximates $L^p$-viscosity solutions to \eqref{eq_main}.

\begin{Proposition}\label{induction2}
	Let $u\in\mathcal{C}(Q_1)$ be a normalized $L^p$-viscosity solution to \eqref{eq_main}. Assume that A\ref{assumption1} and A\ref{assumption5} hold. There exists a sequence of polynomials $(P_n)_{n\in\mathbb{N}}$ of the form
	\[
	P_n(x,t)\coloneqq a_n+b_n\cdot x+c_n\,t+\frac{1}{2}x^td_n\,x
	\]
	such that
	\begin{equation}\label{estimate1}
	\displaystyle\sup_{Q_{\rho^n}}|u(x,t)-P_n(x,t)|\leq \rho^{n(2+\gamma)}
	\end{equation}
	for some $0<\gamma<1$, and
	\[
	c_n+\displaystyle\inf_{\beta \in {\cal B}}(-\tr(\bar A_{\beta}(0,0)d_n))=0,
	\]
	for every $n\geq0$, where 
	\begin{equation}\label{estimate2}
	\footnotesize{|a_n-a_{n-1}|+\rho^{n-1}|b_n-b_{n-1}|+\rho^{2(n-1)}(|c_n-c_{n-1}|+|d_n-d_{n-1}|)\leq C\rho^{(n-1)(2+\gamma)}}
	\end{equation}
	and the constants $C>0$ and $0<\rho<<1$ are universal.	
\end{Proposition}
\begin{proof}
	 Without loss of generality, we assume that $\|u\|_{L^\infty(Q_1)}\leq1$. As in the Proposition \ref{induction1}, we prove this statement by induction in $n\geq 0$. We split the proof in four steps.
	 \medskip
	 
	 \textbf{Step 1}
	 
	 \medskip
	 
	 Let us define
	 \[
	 P_{-1}(x,t)\equiv P_0(x,t)\equiv0.
	 \]
	Hence, the case $n=0$ is obvious. Suppose the case $n=k$ has been verified, for some $k\in\mathbb{N}$. Let us prove the statement for the case $n=k+1$. For that, we introduce the auxiliary function
	\[
	v_k(x,t)\coloneqq\frac{(u-P_k)(\rho^kx, \rho^{2k}t)}{\rho^{k(2+\gamma)}} \ \ \mbox{in} \ \ Q_1
	\]
	which solves the equation
	\[
	(v_k)_t+\frac{1}{\rho^{k\gamma}}\left(\displaystyle\sup_{\alpha \in {\cal A}}\inf_{\beta \in {\cal B}}[-\tr(A_{\alpha,\beta}(\rho^kx,\rho^{2k}t)(\rho^{k\gamma}D^2 v_k+d_k))]+c_k\right)=f_k,
	\] 
where $f_k(x,t):= \rho^{-k\gamma}f(\rho^kx, \rho^{2k}t)$.
	 \medskip
	
	\textbf{Step 2}
	
	\medskip
	 
	 In order to approximate $v_k$ by a suitable function $h\in \mathcal{C}^{2,\bar{\gamma}}_{loc}(Q_1)$, set
	 \[
	 M_k\coloneqq \rho^{k\gamma}M+d_k.
	 \]
	 From the assumption A\ref{assumption5}, it follows that
	 \begin{equation*}
	 \begin{aligned}
	 &\left|\displaystyle\sup_{\alpha \in {\cal A}}\inf_{\beta \in {\cal B}}[-\tr(A_{\alpha,\beta}(\rho^{k}x,\rho^{2k}t)M_k)]-\displaystyle\inf_{\beta \in {\cal B}}[-\tr(\bar A_{\beta}(0,0)M_k)]\right|\vspace{0.2cm}\\
	 &\quad\quad\leq \displaystyle\sup_{\alpha \in {\cal A}}\sup_{\beta \in {\cal B}}\left|\tr[(A_{\alpha,\beta}(\rho^k x,\rho^{2k}t)-\bar A_{\beta}(0,0))M_k]\right|\\
	 &\quad\quad\leq \displaystyle\sup_{(x,t)\in Q_1}\sup_{\alpha \in {\cal A}}\sup_{\beta \in {\cal B}}|A_{\alpha, \beta}(\rho^kx,\rho^{2k}t)-\bar A_{\beta}(0,0)|\|\rho^{k\gamma}M+d_k\| \vspace{0.2cm}\\
	 &\quad\quad\leq \tilde C(d)\varepsilon_3\rho^{\gamma k}(\|\rho^{k\gamma}M+d_k\|)\vspace{0.1cm}\\
	 &\quad\quad\leq \tilde C(d)\varepsilon_3\rho^{\gamma k}(\|M\|+\|d_k\|).
	 \end{aligned}
	 \end{equation*}
	 
	 From the induction hypothesis and the universal choice of $\rho$, we compute 
	 \begin{equation*}
	 \|d_k\| \leq \sum_{i=1}^{k}C\rho^{(i-1)\gamma} \leq \frac{C(1-\rho^{(k-1)\gamma})}{1-\rho^\gamma} \leq \frac{C}{1-\rho^\gamma}\leq C_0.
	 \end{equation*}
	 Then, we find that
	 \begin{equation}\label{assumption scaling}
	 \begin{aligned}
	 &\frac{1}{\rho^{\gamma k}}\left|\displaystyle\sup_{\alpha \in {\cal A}}\inf_{\beta \in {\cal B}}[-\tr(A_{\alpha,\beta}(\rho^{k}x,\rho^{2k}t)M_k)]-\displaystyle\inf_{\beta \in {\cal B}}[-\tr(\bar A_{\beta}(0,0)M_k)]\right|\vspace{0.2cm}\\
	 &\quad\quad\leq \;C_0\tilde C(d)\varepsilon_3(1+\|M\|).
	 \end{aligned}
	 \end{equation}

Also, observe that
\begin{equation*}
\begin{aligned}
\|f_k\|^{p}_{Q_1}&=\frac{1}{\rho^{k\gamma p}}\displaystyle\int_{Q_1}|f(\rho^k x, \rho^{2k}t)|^p dxdt\\
&=\frac{1}{\rho^{k\gamma p}}\displaystyle\dashint_{Q_{\rho^k}}|f(y,s)|^p dyds\\
&\leq \varepsilon_3^p.
\end{aligned}
\end{equation*}

	  Combining the estimate in \eqref{assumption scaling} and standard stability results, given $\delta>0,$ there exists $\varepsilon_3=\varepsilon_3(\delta)$ that ensures the existence of $h\in\mathcal{C}(Q_{3/4})$ satisfying the equation
	 \begin{equation}\label{h equation}
	 \left \{
	 \begin{array}{ll}
	 h_t + \frac{1}{\rho^{k\gamma}}\displaystyle\inf_{\beta \in {\cal B}} [ - {\tr}\left(\bar A_{\beta}(0,0)(\rho^{k\gamma}D^{2}h+d_k)\right)+c_k] = 0  & \text{in} \quad  Q_{ 3/4}, \\
	 h=v_k                                                                                                   & \text{on} \quad  {\partial}Q_{ 3/4}, 
	 \end{array}
	 \right. 
	 \end{equation}
	 such that 
	 \[
	 \|v_k-h\|_{L^{\infty}(Q_{8/9})}\leq\delta.
	 \]
	 
	 Now, let us prove that $h\in\mathcal{C}^{2,\bar{\gamma}}_{loc}(Q_1)$. In fact, first observe that we can rewrite the equation in \eqref{h equation} in the following way
	 \begin{equation}\label{equation h_2}
	 h_t+\displaystyle\inf_{\beta \in {\cal B}}\left[-\tr\left(\bar A_{\beta}(0,0)\left(D^2h+\frac{d_k}{\rho^{k\gamma}}\right)\right)-\frac{c_k}{\rho^{k\gamma}}\right]=0.
	 \end{equation}
	 
	 Because of the Evans-Krylov's parabolic regularity theory, we know that viscosity solutions to 	 
	 \[
	 h_t+\displaystyle\inf_{\beta \in {\cal B}}[-\tr(\bar A_{\beta}(0,0)D^2v)]=0 \ \ \mbox{in} \ \ Q_{3/4}
	 \]
	 are locally of class $\mathcal{C}^{2,\bar{\gamma}}(Q_{3/4})$, for some $\bar\gamma\in(0,1)$, with estimate 
	 \[
	 \|v\|_{\mathcal{C}^{2,\bar{\gamma}}(Q_{1/2})}\leq C,
	 \]
	  for some universal positive constant $C$. 
	 
	 Moreover, from the induction hypothesis we have
	 \begin{equation*}\label{equaion h_1}
	 \frac{c_k}{\rho^{k\gamma}}+\displaystyle\inf_{\beta \in {\cal B}}\left[-\tr\left(\bar A_{\beta}(0,0)\frac{d_k}{\rho^{k\gamma}}\right)\right]=\frac{1}{\rho^{k\gamma}}\left[c_k+\displaystyle\inf_{\beta \in {\cal B}}[-\tr(\bar A_{\beta}(0,0)d_k)]\right]=0.
	 \end{equation*}
	 
	 Therefore, viscosity solutions to 
	 \begin{equation*}
	 v_t+\displaystyle\inf_{\beta \in {\cal B}}\left[-\tr\left(\bar A_{\beta}(0,0)\left(D^2v+\frac{d_k}{\rho^{k\gamma}}\right)\right)+\frac{c_k}{\rho^{k\gamma}}\right]=0
	 \end{equation*}
	 are of class $\mathcal{C}^{2,\bar{\gamma}}(Q_{3/4})$ locally, for some $\bar\gamma\in(0,1)$,  with estimate 
	 \[
	 \|v\|_{\mathcal{C}^{2,\bar{\gamma}}(Q_{1/2})}\leq C,
	 \] 
	 with $C>0$ a universal constant. Combining this fact with \eqref{equation h_2}, we obtain $h\in\mathcal{C}^{2,\bar{\gamma}}_{loc}(Q_1)$, for some $0<\bar\gamma<1$, such that 
	 $\|h\|_{\mathcal{C}^{2,\bar{\gamma}}(Q_{1/2})}\leq C$.
	 Hence, 
\small{	
	 \[
	 \sup_{Q_{\rho}}\left|h(x,t)-\left[h(0,0)+Dh(0,0)\cdot x+h_t(0,0)t+\frac{1}{2}x^tD^{2}h(0,0)x\right]\right|\leq C \rho^{2+\bar\gamma}.
	 \]
	 }
	 \medskip
	 
	 \textbf{Step 3}
	 
	 \medskip
	 
	Setting
	\[
	\bar P(x,t)=h((0,0)+Dh(0,0)\cdot x+h_t(0,0)\,t+\frac{1}{2}x^tD^2h(0,0)x,
	\]
	from the triangular inequality we have
	\begin{equation*}
	\begin{aligned}
	\sup_{Q_{\rho}}|u(x,t)-\bar P(x,t)|&\leq \sup_{Q_{\rho}}|u(x,t)-h(x,t)|+\sup_{Q_{\rho}}|h(x,t)-\bar P(x,t)|\\
	&\leq \delta + C\rho^{2+\bar\gamma}.
	\end{aligned}
	\end{equation*}
	For $0< \gamma < \bar{\gamma}$ fixed, we make the universal choices 
	\[
	\rho:= \left(\dfrac{1}{2C}\right)^{\frac{1}{\bar{\gamma}-\gamma}} \;\;\; \text{and}\;\;\; \delta\coloneqq\frac{\rho^{2+\gamma}}{2}
	\]
to obtain 
\begin{equation}\label{estimatev_k}
\sup_{Q_{\rho}}|v_k(x,t)-\bar P(x,t)|\leq\rho^{2+\gamma}.
\end{equation}
Observe that the universal choice of $\delta$ determines $\varepsilon_3$.	 
		
\medskip

\textbf{Step 4}

\medskip

Set
\[
P_{k+1}(x,t)\coloneqq P_k(x,t)+\rho^{k(2+\gamma)}\bar P(\rho^{-k}x,\rho^{-2k}t).
\]
The estimate in \eqref{estimatev_k} and the definition of $P_{k+1}$ lead to
\[
\displaystyle\sup_{Q_{\rho^{k+1}}} |u(x,t)-P_{k+1}(x,t)|\leq\rho^{(k+1)(2+\gamma)}.
\]

Furthermore, since $c_{k+1}=c_k+\rho^{k\gamma}h_t(0,0)$ and $d_{k+1}=d_k+\rho^{k\gamma}D^2h(0,0)$, from \eqref{h equation} we obtain that
\[
c_{k+1}+\displaystyle\inf_{\beta \in {\cal B}}(-\tr(\bar A_{\beta}(0,0)\,d_{k+1}))=0.
\]

Because of the $\mathcal{C}^{2,\bar{\gamma}}$-estimates for $h$ we have that 
\[
|h(0,0)|+|Dh(0,0)|+|h_t(0,0)|+|D^2h(0,0)|\leq C,
\]
for some universal positive constant $C$. Hence, from the definition of $P_{k+1}$, we can conclude  
\[
|a_{k+1}-a_{k}|+\rho^{k}|b_{k+1}-b_{k}|+\rho^{2k}(|c_{k+1}-c_{k}|+|d_{k+1}-d_{k}|)\leq C\rho^{k(2+\gamma)}.
\]
The proof is now complete.
\end{proof}

\medskip

In what follows, we prove the $\mathcal{C}^{2,\gamma}$-regularity at the origin.

\begin{proof}[Proof of the Theorem \ref{theorem03}]
From the Proposition \ref{induction2}, we can find a polynomial $\bar{P}$ of the form 
\[
\bar{P}(x,t)\coloneqq \bar{a}+\bar{b}\cdot x+\bar{c} \, t+\frac{1}{2}x^t\bar{d}x
\]
such that $P_n \to \bar{P}$ uniformly in $Q_1$.  The regularity of $h$ implies that there exists a constant $C >0$  such that
\[
|D\bar{P}(0,0)| + \|D^2\bar{P}(0.0)\| \leq C,
\]
with the following estimates:
\[
|a_n-\bar{a}|\leq C\rho^{n(2+\gamma)} \,\,\,\,\, ; \,\,\,\,\, |b_n-\bar{b}|\leq C\rho^{n(1+\gamma)}
\]
\[
|c_n-\bar{c}|\leq C\rho^{n\gamma} \,\,\,\,\, \mbox{and} \,\,\,\,\, |d_n-\bar{d}|\leq C\rho^{n\gamma}.
\]
To conclude the proof, given $0<\rho<r$, take the first integer $n\in\mathbb{N}$ satisfying $r^{n+1}<\rho\leq r^n$. Therefore, we can estimate
\begin{equation*}
\begin{array}{rcl}
& &\displaystyle\sup_{Q_\rho}\left|u(x,t)-\bar{P}(x,t)\right|\vspace{0.2cm}\\
& &\quad \leq \displaystyle \sup_{Q_{r^n}}|u(x,t)-P_n(x,t)|+\sup_{Q_{r^n}}|P_n(x,t) - \bar{P}(x,t)|\vspace{0.2cm}\\
& &\quad\leq\displaystyle \frac{C}{r}r^{(n+1)(2+\gamma)}\vspace{0.2cm}\\
& &\quad\leq \displaystyle C\rho^{2+\gamma}.
\end{array}
\end{equation*}
This finishes the proof of the theorem.
\end{proof}

\begin{Remark}\label{last_rem}
Even though we only prove pointwise estimates at the origin,  Theorem \ref{theorem03} holds true at any point that satisfies assumption A$\ref{assumption5}$.  Also,  notice that we recover the classical ${\mathcal C}^{2,\gamma}$ theory for Bellman equations when assumption A$\ref{assumption5}$ is satisfied for every point in $Q_1$,  since we would have $A_{\alpha,\beta}(x,t) = \bar{A}_{\beta}(x,t)$ for every $(x,t) \in Q_1$  and $\bar{A}_\beta$ is of class ${\mathcal C}^\gamma$.  
\end{Remark}

\medskip 

\noindent{\bf Acknowledgements:}  PA is partially supported by FAPERJ (Grant \# E-26/ 201.609/2018) and CAPES - Brazil. GR is also supported by CAPES - Brazil. MS is supported by the CONACyT-Mexico. This study was financed in part by the Coordenação de Aperfeiçoamento de Pessoal de Nível Superior - Brazil (CAPES) - Finance Code 001.

\bigskip

\bibliography{bib_2018}

\bibliographystyle{plain}

\bigskip

\noindent\textsc{P\^edra D. S. Andrade}\\
Department of Mathematics\\
Pontifical Catholic University of Rio de Janeiro -- PUC-Rio\\
22451-900, G\'avea, Rio de Janeiro-RJ, Brazil\\
\noindent\texttt{pedra.andrade@mat.puc-rio.br}

\vspace{.15in}

\noindent\textsc{Giane C. Rampasso}\\
Department of Mathematics\\
University of Campinas -- IMECC -- Unicamp\\
13083-859, Cidade Universitária, Campinas-SP, Brazil\\
\noindent\texttt{girampasso@ime.unicamp.br}

\vspace{.15in}

\noindent\textsc{Makson S. Santos}\\
Centro de Investigación en Matemáticas - CIMAT\\
36023, Valenciana, Guanajuato, Gto, Mexico.\\
\noindent\texttt{makson.santos@cimat.mx}

\end{document}